\documentclass[12pt, a4paper, abstracton, bibliography=totoc]{scrartcl}
\pdfoutput=1

\usepackage{amsmath, amsthm, amsfonts, amssymb}
\usepackage[utf8]{inputenc}
\usepackage{color}
\usepackage{slashed} 
\usepackage[all, cmtip]{xy}
\usepackage{hyperref} 

\let\counterwithout\relax

\usepackage{chngcntr} 

\usepackage{ifthen}
\usepackage[nouppercase]{scrpage2}

\automark[section]{section}          
\deftripstyle{myheadings}
             {\ifthenelse{\isodd{\value{page}}}{\leftmark}{\pagemark}}
             {}
             {\ifthenelse{\isodd{\value{page}}}{\pagemark}{\rightmark}}%
             {}{}{}%
\usepackage{graphicx} 
\usepackage[english]{babel}
\usepackage{microtype}
\usepackage{bm} 
\usepackage{xfrac} 
\usepackage{mathtools} 
\usepackage{mathabx}

\usepackage{tikz}
\usetikzlibrary{patterns}
\def\gvertx#1{%
  \fill #1 circle (.1cm);}
\def\egvertx#1{%
  \filldraw[fill=white] #1 circle (.1cm);}

\newcommand{\IN}{\mathbb{N}}
\newcommand{\IZ}{\mathbb{Z}}

\newcommand{\IR}{\mathbb{R}}
\newcommand{\IC}{\mathbb{C}}

\newcommand{\IB}{\mathfrak{B}}

\newcommand{\diam}{\operatorname{diam}}

\newcommand{\id}{\operatorname{id}}

\newcommand{\ch}{\operatorname{ch}}

\DeclareMathOperator{\fin}{fin}

\newcommand{\fa}[1]{\forall_{#1}\quad}

\newtheorem{thm}{Theorem}[section]
\newtheorem*{thm*}{Theorem}
\newtheorem*{mainthm*}{Main Theorem}
\newtheorem{cor}[thm]{Corollary}
\newtheorem*{cor*}{Corollary}
\newtheorem*{maincor*}{Corollary}
\newtheorem{lem}[thm]{Lemma}
\newtheorem{prop}[thm]{Proposition}

\newtheorem*{question*}{Question}
\theoremstyle{definition}
\newtheorem{rem-alt}[thm]{Remark}
\newtheorem*{rem*}{Remark}
\newtheorem{ex-alt}[thm]{Example}

\newtheorem*{example*}{Example}
\newtheorem*{examples*}{Examples}
\newtheorem{defn-alt}[thm]{Definition}

\newenvironment{defn}    
{%
	\pushQED{\qed}\begin{defn-alt}}
	{\popQED\end{defn-alt}}
	
{%
	\pushQED{\qed}\begin{ex-alt}}
	{\popQED\end{ex-alt}}

\newenvironment{rem}    
{%
	\pushQED{\qed}\begin{rem-alt}}
	{\popQED\end{rem-alt}}

\numberwithin{equation}{section}
\counterwithout{footnote}{section}

\makeatletter
\def\blfootnote{\gdef\@thefnmark{}\@footnotetext}
\makeatother

\begin{document}

\title{Polynomially weighted $\ell^p$-completions and group homology}
\author{Alexander Engel\thanks{\href{mailto:alexander.engel@mathematik.uni-regensburg.de}{alexander.engel@mathematik.uni-regensburg.de}}
\and
Clara Löh\thanks{\href{mailto:clara.loeh@mathematik.uni-regensburg.de}{clara.loeh@mathematik.uni-regensburg.de}}
}
\date{\small{Fakult\"{a}t f\"{u}r Mathematik\\ Universit\"{a}t Regensburg\\ 93040 Regensburg, Germany}}
\maketitle

\begin{abstract}
  We introduce polynomially weighted $\ell^p$-norms on the bar complex of a finitely generated group. We prove that, for groups of polynomial or exponential growth, the homology of the completed complex does not depend on the value of~$p$ in the range $(1,\infty)$. 
\end{abstract}

\tableofcontents

\section{Introduction}

Let $G$ be a finitely generated group. In order to state our main result, we quickly introduce the main players in it.

The group homology $H_\ast(G;\IC)$ can be computed as the homology of
the bar complex~$C_*(G;\IC)$. Chains~$c \in C_k(G;\IC)$ are of the form
$c = \sum_{g \in G^{k}} a_g \cdot [e,g_1,\dots, g_k]$,
where only finitely many of the coefficients $a_g$ are non-zero. We
choose a finite generating set $S$ for $G$ to get a word-metric on
$G$. For~$n \in \IN$ and $p \in [1,\infty)$ we then define a weighted norm on
  $C_k(G;\IC)$ by $\|c\|_{n,p}^S := \bigl(\sum_{g\in G^k} |a_g|^p \cdot
  \diam_S(g)^n\bigr)^{1/p}$. We equip $C_k(G;\IC)$ with the
  family~$(\|-\|_{n,p}^S + \|\partial - \|_{n,p}^S)_{n\in \IN}$ of
  norms and denote the corresponding completion to a Fr\'echet space
  by~$C_k^p(G)$. The homology of the resulting chain
  complex~$C_*^p(G)$ is denoted by~$H_*^p(G)$.

\begin{mainthm*}[Theorem~\ref{thm:pqcompgeneral}, Proposition~\ref{prop:comppq}]
  Let $G$ be a finitely generated group of polynomial or exponential growth
  and let $p,q \in (1,\infty)$ with~$p < q$.
  Then the canonical homomorphism~$H_*^p(G) \longrightarrow H_*^q(G)$
  is an isomorphism.
\end{mainthm*}

\paragraph{Relation to the strong Novikov conjecture}
Let us explain why we are interested in a theorem like the one above. We first have to recall the following two notions.
Firstly, a group $G$ is called of \emph{type~$F_\infty$}, if it admits a model for its classifying space~$BG$
of finite type (i.e., a CW-complex that in each dimension has only finitely many cells).
Secondly, a group of type~$F_\infty$ is called \emph{polynomially contractible}, if its Dehn function and its higher-dimensional analogues are polynomially bounded. Note that this assumption on $G$ is not very strong: most of the groups that one would call non-positively curved (like hyperbolic groups, CAT(0)-groups, systolic groups or mapping class groups) are polynomially contractible. This follows from the fact that if a group is polynomially combable (i.e., combable with a uniform polynomial bound on the lengths of the combing paths), then it is polynomially contractible~\cite[End of 2nd paragraph on p.~257]{ji_ramsey}\cite[Prop.~3.4]{engel_BSNC}.

For hyperbolic groups any choice of geodesics in the Cayley graph will be a suitable combing, for CAT(0)-groups any choice of quasi-geodesics in the group following uniformly closely CAT(0)-geodesics in the underlying space will do the job, for systolic groups one can use the bi-automatic structure found by Januszkiewicz--\'{S}wi{\c{a}}tkowski~\cite[Thm.~E]{janus_swia} or the combing by Osajda--Przytycki~\cite{osajda_przytycki}, and an automatic structure on mapping class groups was provided by Mosher~\cite{mosher}. For a thorough compilation of polynomially contractible groups see the introduction of~\cite{engel_BSNC}.

\begin{maincor*}
  Let $G$ be a group of type $F_\infty$, and of polynomial or exponential growth,
  and let $G$ be polynomially contractible. 
If there exists some $p \in (1,\infty)$ such that the canonical homomorphism~$H^1_\ast(G) \to H^p_\ast(G)$ is injective, then the strong Novikov conjecture holds for $G$.
\end{maincor*}

\begin{proof}
The proof relies on the following diagram~\cite{engel_BSNC}: 
\[\xymatrix{
RK_k(BG) \ar[r] \ar[d]_{\ch_k} & K_k(B_r^p G) \ar@{-->}[d]\\
H_k(G;\IC) \ar[r] & H_k^1(G)
}\]
Here, $B_r^p G$ denotes the norm completion of $\IC G \subset \IB(\ell^p G)$ and the top horizontal map is the analytic assembly map. In the case $p=2$ we have $B_r^2 G = C_r^\ast G$, i.e., the reduced group $C^\ast$-algebra, and the \emph{strong Novikov conjecture} asserts that the analytic assembly map in this case (i.e., for $p=2$) is rationally injective.

Let $x \in RK_\ast(BG) \otimes \IC$ be any non-trivial element. Because the homological Chern character $\ch_\ast\colon RK_\ast(BG) \otimes \IC \to H_\ast(G;\IC)$ is an isomorphism, there is some~$k \in \IN$ such that $\ch_k(x) \in H_k(G;\IC)$ is non-zero. If $G$ is of type $F_\infty$ and polynomially contractible, then the canonical map $H_\ast(G;\IC) \to H_\ast^1(G)$ is an isomorphism~\cite[Corollary 4.4]{engel_BSNC}. (Analogous statements in the dual situation, i.e., for the corresponding cohomology groups, also hold~\cite{ogle_pol,ji_ramsey,meyer}.)
Further,  the right vertical map in the above diagram exists for all~$p \le (k+2) / (k+1)$. Hence, for such $p$ the element $x$ is not in the kernel of the analytic assembly map.

Since our goal is the strong Novikov conjecture, i.e., to show that the element $x$ is not in the kernel of the assembly map for $p=2$, we can try to go with the lower horizontal map to $H_k^q(G)$ for
some~$q >1$ instead of to~$H^1_k(G)$, i.e, we consider the new diagram
\[\xymatrix{
RK_\ast(BG) \ar[rr] \ar[d]_{\ch_k} & & K_\ast(B_r^p G) \ar@{-->}[d]\\
H_k(G;\IC) \ar[r] & H_k^1(G) \ar[r] & H_k^q(G)
}\]

Also in this case, we can construct the right vertical map for all $p \le q \cdot (k+2) / (k+1)$
(this can be shown as in previous work of the first named author~\cite[Proposition~5.1]{engel_BSNC}).
In particular, if $q$ is big enough, we can do it for $p=2$. We have already noted above that polynomial contractibility gives us that the canonical map $H_\ast(G;\IC) \to H_\ast^1(G)$ is an isomorphism. Using the main theorem, we see that if the canonical map~$H_\ast^1(G) \to H_\ast^p(G)$ is injective for some $p > 1$, then $H_\ast^1(G) \to H_\ast^q(G)$ will be injective for every $q \in (1,\infty)$. Hence our element $x$ is not in the kernel of the assembly map for the case $p=2$.
\end{proof}

Unfortunately, the hypotheses of this corollary are \emph{not} satisfied for all groups:
For example, for the free group~$F_2$ of rank~$2$, the canonical homomorphism~$H^1_1(F_2) \longrightarrow H^p_1(F_2)$ is trivial for all~$p \in (1,\infty)$ (Theorem~\ref{thm:f2vanishing}, Theorem~\ref{thm:p=1}).
In fact, we expect this vanishing result to hold in far greater generality, and thus this approach to the strong Novikov conjecture is not promising.

\paragraph{Questions}
Let us collect some open problems arising from the present paper. Since this seems to be the first time that such polynomially weighted $\ell^p$-completions of group homology are defined, there are many natural questions left open.
  \begin{itemize}
  \item
    Does Theorem~\ref{thm:pqcompgeneral}, i.e., the comparison in the range $(1,\infty)$, also hold for groups
    of intermediate growth?
  \item
    For which groups of superpolynomial growth does Theorem~\ref{thm:pqcompgeneral} also
    hold in the cases~``$p=1$'' or~``$q= \infty$''\;?
  \item
    For which groups~$G$ and which~$p \in (1,\infty]$ is the canonical
    map~$H^1_*(G) \longrightarrow H^p_*(G)$, resp.~the canonical map $H_*(G;\IC) \to H^p_*(G)$, injective?
  \item
    For which groups~$G$ and which~$p \in [1,\infty]$, $k \in \IN$ is~$H_k^p(G)$ non-trivial?
    How can such classes be detected? 
  \end{itemize}

\paragraph{Related work}
Though this seems to be the first time that these polynomially weighted $\ell^p$-completions of group homology are defined, there are of course similar things already in the literature:

\begin{itemize}
\item Bader, Furman and Sauer~\cite{bfs} investigate the comparison maps from ordinary homology and Sobolev homology, respectively, to the $\ell^1$-homology of any word hyperbolic group.
\item Nowak and {\v{S}}pakula~\cite{nowak_spakula} study coarse homology theory with prescribed growth conditions.
\item Weighted simplicial homology was studied by Dawson~\cite{weighted_simpl_complexes} and by Ren, Wu and Wu~\cite{ren_wu_wu}.
\item The dual situation to the one from the present paper, but only in the case of~$\ell^1$, i.e., group cohomology of polynomial growth, was studied by Connes and Moscovici~\cite{connes_moscovici} in relation with the strong Novikov conjecture, and further investigated by many others like Ji~\cite{ji}, Meyer~\cite{meyer} and Ogle~\cite{ogle_pol}.
\end{itemize}

\paragraph{Overview of this article}
Section~\ref{sec:lpnorms} introduces the polynomially weighted
$\ell^p$-versions of group homology in full detail and
discusses the case of groups of polynomial growth. In
Section~\ref{sec:comparison}, we establish the comparison theorem for
groups of exponential growth.  The vanishing result for the free group
is proved in Section~\ref{sec:vanishing}.

\paragraph{Acknowledgements}
The authors were supported by the SFB 1085 \emph{Higher Invariants} of the Deutsche Forschungsgemeinschaft DFG. The first named author was also supported by the Research Fellowship EN 1163/1-1 \emph{Mapping Analysis to Homology} of the DFG, and the DFG Priority Programme SPP 2026 \emph{Geometry at Infinity} (EN 1163/3-1 ``Duality and the coarse assembly map''). Moreover, we would like to thank the anonymous referee for providing helpful comments.

\section{Weighted \texorpdfstring{${\ell^p}$}{lp}-norms on group homology}\label{sec:lpnorms}

\subsection{Definition and basic properties}

\begin{defn}[weighted $\ell^p$-norms]
  Let $G$ be a finitely generated group endowed with a finite generating set~$S$, let $k \in \IN$, and let $p
  \in [1,\infty)$. For~$n \in \IN$ we define the \emph{$n$-weighted
      $\ell^p$-norm (with respect to~$S$)} by 
  \begin{align*}
    \|-\|_{n,p}^S \colon C_k(G;\IC) & \longrightarrow \IR_{\geq 0} \\
    \sum_{g \in G^{k}} a_g \cdot [e,g_1,\dots, g_k]
    & \longmapsto \biggl(\sum_{g\in G^k} |a_g|^p \cdot \diam_S(g)^n\biggr)^{1/p},
  \end{align*}
  where $\diam_S(g) := \diam_S\{e,g_1,\dots, g_k\}$ is the diameter with respect
  to the word metric~$d_S$ on~$G$.
  
  We then equip~$C_k(G;\IC)$ with the family~$(\|-\|_{n,p}^S +
  \|\partial - \|_{n,p}^S)_{n\in \IN}$ of norms and denote the
  corresponding completion to a Fr\'echet space by~$C_k^p(G)$. By
  construction, the boundary operator of~$C_k(G;\IC)$ extends continuously
  to~$C_k^p(G)$ and the homology of~$C_*^p(G)$ is called
  \emph{$\ell^p$-polynomially bounded homology of~$G$}, denoted
  by~$H_k^p(G)$.

  In the case of~$p = \infty$, we proceed in the same manner, using
  the \emph{$n$-weighted $\ell^\infty$-norms (with respect to~$S$)}, 
  defined by
  \begin{align*}
    \|-\|_{n,\infty}^S \colon C_k(G;\IC) & \longrightarrow \IR_{\geq 0} \\
    \sum_{g \in G^{k}} a_g \cdot [e,g_1,\dots, g_k]
    & \longmapsto \sup_{g \in G^k} |a_g| \cdot \diam_S(g)^n.
\qedhere
  \end{align*}
\end{defn}

\begin{rem}
  If $G$ is a finitely generated group, $k, n \in \IN$, and $p \in
  [1,\infty]$, then different finite generating sets~$S,T$ of~$G$ lead
  to equivalent (semi-)norms~$\|-\|_{n,p}^S$ and $\|-\|_{n,p}^T$
  on~$C_k(G;\IC)$. Therefore, the completions~$C_*^p(G)$ and the
  homology~$H_*^p(G)$ are independent of the choice of finite
  generating sets.
\end{rem}

\begin{rem}\label{rem:p<qmap}
  Let $G$ be a finitely generated group and let $p,q \in [1,\infty)$ with~$p < q$. Then
  the canonical inclusion~$C_*^p(G) \longrightarrow C_*^q(G)$ is
  contractive in the following sense: For every finite generating set~$S$ of~$G$
  and every~$n \in \IN$ the identity map~$C_*(G;\IC) \longrightarrow C_*(G;\IC)$ has norm
  at most~$1$ with respect to the norms~$\|-\|_{n,p}^S$ and $\|-\|_{n,q}^S$,
  respectively. In particular, we obtain a canonical induced map
  \[ H_*^p(G) \longrightarrow H_*^q(G).
  \]
  If $p \in [1,\infty)$ and $n \in \IN$, then
  \[ \fa{c \in C_k(G;\IC)} \|c\|^S_{n, \infty} \leq \|c\|_{\lceil n \cdot p\rceil,p}^S,
  \]
  which yields a canonical map~$H_*^p(G) \longrightarrow H_*^\infty(G)$.
\end{rem}

\subsection{The case \texorpdfstring{$p=1$}{p=1}}

It is already known that the canonical map~$H_*(G;\IC) \longrightarrow H_*^1(G)$ is an isomorphism for a large class of groups.

To state the corresponding theorem, we have to recall two notions.
Firstly, a group $G$ is called of \emph{type~$F_\infty$}, if it admits a model for its classifying space~$BG$
of finite type (i.e., a CW-complex that in each dimension has only finitely many cells).
Secondly, a group of type~$F_\infty$ is called \emph{polynomially contractible}, if its Dehn function and its higher-dimensional analogues are polynomially bounded.

Most of the groups that one calls non-positively curved (like hyperbolic groups, systolic groups, $\operatorname{CAT(0)}$-groups or mapping class groups) are polynomially contractible (see the introduction for references).

The following theorem has been proved (in variations) by different people in different ways~\cite{engel_BSNC,connes_moscovici,meyer,ogle_pol,ji_ramsey,JOR_B_bounded_coho}:

\begin{thm}\label{thm:p=1}
  Let $G$ be a group of type~$F_\infty$ that is polynomially contractible. Then
  the canonical map~$H_*(G;\IC) \longrightarrow H_*^1(G)$ is an isomorphism.
\end{thm}

\begin{rem}
Without the assumption of polynomial contractibility, Theorem~\ref{thm:p=1} is likely false.

Ji, Ogle, and Ramsey provided groups whose comparison maps from bounded cohomology to ordinary cohomology fail to be injective or surjective, respectively~\cite[Sec.~6.4 \& 6.5]{JOR_B_bounded_coho}. Bounded cohomology, as they investigate it, is dual to our~$H_*^1(-)$ in the sense that it pairs with it (and this pairing is compatible with the usual pairing of homology with cohomology under the comparison maps). Please be also aware that their bounded cohomology is \emph{not} the one used by Gromov~\cite{gromov_vol}.

It seems therefore plausible that the groups of Ji, Ogle and Ramsey are also examples of groups for which the canonical map~$H_*(G;\IC) \longrightarrow H_*^1(G)$ is not injective or surjective, respectively. 
\end{rem}

\subsection{Comparison on groups of polynomial growth}

\begin{prop}\label{prop:comppq}
  Let $G$ be a finitely generated group of polynomial growth and let
  $p,q \in [1,\infty]$ with~$p < q$.
  \begin{enumerate}
    \item
      Then the inclusion~$C_*^p(G) \longrightarrow C_*^q(G)$ is bounded
      from below in the following sense: For every finite generating set~$S$
      of~$G$ and all~$k,n \in \IN$ there exist~$C \in \IR_{>0}$ and~$m \in \IN$ with
      \[ \fa{c \in C_k(G;\IC)} \|c\|_{m,q}^S \geq C \cdot \|c\|_{n,p}^S. 
      \]
    \item
      In particular, the canonical map~$H_*^p(G) \longrightarrow H_*^q(G)$
      (Remark~\ref{rem:p<qmap}) is an isomorphism.
  \end{enumerate}
\end{prop}

\begin{proof}
  \emph{Ad~1.}
  We first consider the case $q < \infty$.
  Let $D \in \IN$ be the polynomial growth rate of~$G$, let $S$ be a
  finite generating set of~$G$, and let $k, n \in \IN$. Then
  \[ m := \Bigl\lceil q \cdot \Bigl((k \cdot D+2) \cdot
           \Bigl(\frac1p - \frac1q\Bigr) + \frac np \Bigr)\Bigr\rceil
  \]
  has the desired property, as can be seen by the generalized H\"older
  inequality: Because $D$ is the polynomial growth rate of~$G$, there
  is a constant~$K \in \IR_{>0}$ with
  \[ \fa{r \in \IN_{>0}} \beta(r) := \bigl|B_e^{G,S}(r)\bigr| \leq K \cdot r^D.
  \]
  Moreover, because of~$q > p$, there is~$q' \in [1,\infty)$ with
  \[ \frac1q + \frac1{q'} = \frac1p.
  \]
  We now consider $c \in C_k(G;\IC)$ and the weight functions
  \begin{align*}
    w_1,w_2 \colon G^k & \longrightarrow \IR_{\geq0} \\
    w_1 : g & \longmapsto \diam_S(g)^{m/q}\\
    w_2 : g & \longmapsto \diam_S(g)^{n/p-m/q}.
  \end{align*}
  By definition, $\|c\|_{n,p}^S = \| c \cdot w_1 \cdot w_2\|_p$
  and $\|c\|_{m,q}^S = \|c \cdot w_1\|_q$ 
  (where
  ``$\cdot$'' denotes pointwise multiplication). Applying the
  generalized H\"older inequality, we hence obtain
  \begin{align*}
    \|c\|_{n,p}^{S}
    & \leq \| c \cdot w_1 \|_q \cdot \|w_2\|_{q'}
    = \| c \|_{m,q}^S \cdot \|w_2\|_{q'}
  \end{align*}
  and it remains to bound~$\|w_2\|_{q'}$ by a constant. The polynomial growth
  condition yields
  \begin{align*}
    \bigl(\|w_2\|_{q'}\bigr)^{q'}
    & = \sum_{g \in G^k} \diam(g)^{q' \cdot (\frac np - \frac mq)}
    \leq \sum_{r =1}^\infty \beta(r)^k
    \cdot r^{q' \cdot (\frac np - \frac mq)}
    \leq K^k \cdot \sum_{r=1}^\infty r^{k \cdot D}
    \cdot r^{q' \cdot (\frac np - \frac mq)}\\
    & \leq K^k \cdot \sum_{r=1}^{\infty} r^{-2}
    = K^k \cdot\zeta(2).
  \end{align*}

  In the case~$q = \infty$, one can proceed in a similar way
  (with~$m = \lceil 1/p \cdot (k \cdot D + n +2)\rceil$).
  
  \emph{Ad~2.} 
  By the first part, the identity map on the ordinary chain complex~$C_*(G;\IC)$ induces an
  isomorphism~$C_*^p(G) \longrightarrow C_*^q(G)$. 
  Hence, the claim follows.
\end{proof}

\subsection{Functoriality of weighted \texorpdfstring{$\ell^p$}{lp}-chain complexes}

\begin{defn}[polynomially controlled kernel]
  Let $G$ be a finitely generated group. 
  \begin{itemize}
  \item A subgroup~$H \subset G$ is \emph{polynomially controlled}, if
    for one (whence every) finite generating set~$S \subset G$ there
    exist $D \in \IN$ and $K \in \IR_{>0}$ such that
    \[ \fa{r \in \IN_{>0}}
    \bigl| B_r^{G,S}(e) \cap H \bigr|
    \leq K \cdot r^D.
    \]
  \item Let $G'$ be a group. 
    A group homomorphism~$\varphi \colon G \longrightarrow G'$ has
    \emph{polynomially controlled kernel} if the subgroup~$\ker \varphi$ of $G$
    is polynomially controlled in the sense above.
\qedhere
  \end{itemize}
\end{defn}

Clearly, all group homomorphisms with finite kernel have polynomially
controlled kernel as well as all group homomorphisms mapping out of
finitely generated groups of polynomial growth. 

\begin{rem}\label{rem_pol_controlled_fin_generated}
If $H$ is a polynomially controlled subgroup of a finitely generated group~$G$, then every finitely generated subgroup $K$ of $H$ has polynomial growth.

The reason for this is that the inclusion $K \to G$ does not increase lengths of elements if we choose a finite generating set $S$ of $G$ containing the chosen finite generating set $T$ of $K$ to define the word lengths. More concretely, we have for all $r \in \IN_{>0}$
\[
|B_r^{K,T}(e)| \le |B_r^{G,S}(e) \cap K| \le |B_r^{G,S}(e) \cap H|.
\]
We see that we even actually have that the growth rates of finitely generated subgroups of $H$ are uniformly bounded from above.
\end{rem}

\begin{lem}\label{lem:hyppolysub}
  Let $G$ be a hyperbolic group and $H$ a subgroup of $G$. 
  Then $H$ is a polynomially controlled subgroup if and only if $H$ is virtually cyclic.
\end{lem}

\begin{proof}
  Let $H$ be a polynomially controlled subgroup. By Remark~\ref{rem_pol_controlled_fin_generated} finitely
  generated subgroups of~$H$ are of polynomial growth. In particular,
  $H$ does not contain a free group of rank~$2$.  As the ambient
  group~$G$ is hyperbolic, this implies that $H$ is virtually
  cyclic~\cite[Corollaire on p.~224]{ghys}.
  
  Let $H$ be virtually cyclic. Without loss of generality we can assume that $H$ is isomorphic to $\IZ$. Then $H$ is quasi-isometrically embedded in $G$ \cite[Corollary III.$\Gamma$.3.10(1)]{bridson_haefliger} and hence a polynomially controlled subgroup of $G$.
\end{proof}

\begin{prop}\label{prop:polkernel}
  Let $p \in [1,\infty]$, let $G$ and $H$ be finitely generated
  groups, and let $\varphi \colon G \longrightarrow H$ be a group
  homomorphism with polynomially controlled kernel.
  Then the induced chain map~$C_*(\varphi;\IC) \colon C_*(G;\IC)
  \longrightarrow C_*(H;\IC)$ is continuous with respect to the
  weighted $\ell^p$-Fr\'echet topologies.
\end{prop}
\begin{proof}
  Let us establish some notation: 
  Let $S \subset G$ and $T \subset H$ be finite generating sets, and
  without loss of generality we may assume that $\varphi(S) \subset T$.
  Let $D \in \IN$ be the polynomial control rate of~$\ker \varphi$; hence,
  there is a~$K \in \IR_{>0}$ with
  \[ \fa{r \in \IN_{>0}}
     \beta(r) := \bigl| B_e^{G, S}(r) \cap \ker \varphi \bigr|
     \leq K \cdot r^D. 
  \]   
  Furthermore, let $p \in [1,\infty]$ and $k,n \in \IN$. It then
  suffices to prove that there exist $m \in \IN$ and $C \in \IR_{>0}$
  such that
  \[ \fa{c \in C_k(G;\IC)}
     \bigl\| C_k(\varphi;\IC)(c) \bigr\|^T_{n,p}
     \leq C \cdot \|c\|_{m+n,p}^S.
  \]
  The arguments are similar to the proof of Proposition~\ref{prop:comppq}:

  Let $c \in C_k(G;\IC)$, say~$c = \sum_{g \in G^k} a_g \cdot [e,g_1,\dots,g_k]$.
  By definition of~$C_k(\varphi;\IC)$, we have 
  \[ \varphi(c)
     := C_k(\varphi;\IC)(c)
     = \sum_{h \in H^k} \biggl(\sum_{g \in \varphi^{-1}(h)} a_g\biggr) \cdot [e,h_1,\dots, h_k],
  \]
  where $\varphi^{-1}(h) := \bigl\{ g \in G^k \bigm|
  \fa{j \in \{1,\dots,k\}} \varphi(g_j) = h_j\bigr\}$.    

  We will first consider the case~$p \in (1, \infty)$. Let 
  \[ m := \bigl\lceil (k \cdot D + 2) \cdot (p-1)\bigr\rceil
  \]
  and let $\overline p := p / (p-1)$. As first step, we bound the
  inner sum for a given~$h \in H^k$ (without loss of generality, we
  may assume~$h \neq (e,\dots,e)$): By the H\"older inequality,
  \begin{align*}
    \biggl| \sum_{g \in \varphi^{-1}(h)} a_g \biggr|^p
    & \leq \biggl( \sum_{g \in \varphi^{-1}(h)} |a_g|\biggr)^p
    \\
    &
    \leq \sum_{g \in \varphi^{-1}(h)} |a_g|^p \cdot \diam_S(g)^m
    \cdot \biggl( \sum_{g \in \varphi^{-1}(h)} \diam_S(g)^{- \frac{\overline p \cdot m}{p}}\biggr)^{\frac{p}{\overline p}}
    \\
    &
    \leq \sum_{g \in \varphi^{-1}(h)} |a_g|^p \cdot \diam_S(g)^m
    \cdot \biggl( \sum_{g \in \varphi^{-1}(h)} \diam_S(g)^{- \frac{m}{p-1}}\biggr)^{p-1}.
  \end{align*}
  The first factor is related to~$\|c\|_{m,p}^S$ and hence of the
  right type.  We will now take care of the second factor: To this
  end, for~$j \in \{1,\dots,k\}$ let $g_j(h) \in G$ be a minimiser
  of~$\min \bigl\{ d_S(e,g) \bigm| g \in G,\ \varphi(g) = h_j\bigr\}$.
  Then we have
  \begin{align*}
    \fa{\kappa_j \in \ker \varphi} d_S\bigl(e, g_j(h) \cdot \kappa_j)
    & \geq \frac12 \cdot \bigl (d_S(e,g_j(h)) + d_S(e, g_j(h) \cdot \kappa_j) \bigr)
    \\
    &
    \geq \frac12 \cdot d_S\bigl(g_j(h), g_j(h) \cdot \kappa_j\bigr)
    = \frac12 \cdot d_S(e,\kappa_j)
  \end{align*}
  and hence the polynomial control on the kernel yields
  \begin{align*}
    \sum_{g \in \varphi^{-1}(h)} \diam_S(g)^{- \frac{m}{p-1}}
    & = \sum_{\kappa \in (\ker \varphi)^k} \diam_S\bigl(g(h) \cdot \kappa\bigr)^{-\frac{m}{p-1}}
    \\
    & 
    \leq \sum_{\kappa \in (\ker \varphi)^k} \bigl(\max_{j \in \{1,\dots,k\}} d_S(e, g_j(h) \cdot \kappa_j)\bigr)^{-\frac{m}{p-1}}
    \\
    &
    \leq 2^{\frac m{p-1}} \cdot 
    \sum_{\kappa \in (\ker \varphi)^k} \bigl( \max_{j \in \{1,\dots,k\}} d_S(e,\kappa_j)\bigr)^{-\frac{m}{p-1}}
    \\
    &
    \leq 2^{\frac m{p-1}} \cdot 
    \sum_{r=1}^\infty \beta(r)^k \cdot r^{-\frac{m}{p-1}}
    \leq 2^{\frac m{p-1}} \cdot K^k \cdot 
    \sum_{r=1}^\infty r^{k\cdot D} \cdot r^{-\frac{m}{p-1}}
    \\
    &
    \leq 2^{\frac m{p-1}} \cdot K^k \cdot \zeta(2).
  \end{align*}
  We set~$C := (2^{\frac m{p-1}} \cdot K^k \cdot \zeta(2))^{1/(p-1)}$. 
  Putting it all together, we obtain (because~$\varphi(S) \subset T$)
  \begin{align*}
    \bigl(\|\varphi(c)\|_{n,p}^T\bigr)^p
    & = \sum_{h \in H^k} \biggl| \sum_{g \in \varphi^{-1}(h)} a_g \biggr|^p \cdot \diam_T(h)^n
    \\
    &
    \leq C \cdot \sum_{h \in H^k} \sum_{g \in \varphi^{-1}(h)} |a_g|^p \cdot \diam_S(g)^m \cdot \diam_T(h)^n
    \\
    & 
    \leq C \cdot \sum_{h \in H^k} \sum_{g \in \varphi^{-1}(h)} |a_g|^p \cdot \diam_S(g)^m \cdot \diam_S(g)^n
    \\
    &
    \leq C \cdot \sum_{g \in G^k} |a_g|^p \cdot \diam_S(g)^{m+n}
    = C \cdot \bigl(\|c\|_{m+n,p}^S\bigr)^p,
  \end{align*}
  as desired.

  In the case $p=1$, the estimates above simplify significantly
  because the inner sum can be treated directly with the inherited
  $\ell^1$-bound and one obtains
  \[ \fa{c \in C_k(G;\IC)} \bigl\| \varphi(c)\bigr\|_{n,1}^T \leq \|c\|_{n,1}^S.
  \]

  In the case $p=\infty$, we take~$m := k \cdot D + 2$. Then the inner sum
  admits the following estimate for given~$h \in H^k$ (without loss of generality, we
  may assume~$h \neq (e,\dots,e)$):
  \begin{align*}
    \biggl| \sum_{g \in \varphi^{-1}(h)} a_g \biggr|
    &
    \leq \sum_{g \in \varphi^{-1}(h)} |a_g| \cdot \diam_S(g)^m \cdot \diam_S(g)^{-m}
    \\
    &
    \leq \sup_{g \in \varphi^{-1}(h)} |a_g| \cdot \diam_S(g)^m
    \cdot \sum_{g \in \varphi^{-1}(h)} \diam_S(g)^{-m}
    \\
    & 
    \leq \sup_{g \in \varphi^{-1}(h)} |a_g| \cdot \diam_S(g)^m \cdot 2^m \cdot K^k \cdot \zeta(2).
  \end{align*}
  This implies~$\| \varphi(c)\|_{n,\infty}^T \leq 2^m \cdot K^k \cdot \zeta(2) \cdot \|c\|_{m+n,\infty}^S$.
\end{proof}

\begin{cor}[functoriality]
  Let $p \in [1,\infty]$, let $G$ and $H$ be finitely generated
  groups, and let $\varphi \colon G \longrightarrow H$ be a group
  homomorphism with polynomially controlled kernel.
  \begin{enumerate}
  \item
    Then $C_*(\varphi;\IC)$ admits a well-defined, continuous extension
    \[ C_*^p(\varphi) \colon C_*^p(G) \longrightarrow C_*^p(H),
    \]
    which is a chain map.
  \item
    In particular, we obtain a corresponding
    homomorphism~$H_*^p(\varphi) \colon H_*^p(G) \longrightarrow
    H_*^p(H)$ that is compatible with~$H_*(\varphi;\IC) \colon
    H_*(G;\IC) \longrightarrow H_*(H;\IC)$.
  \end{enumerate}
\end{cor}
\begin{proof}
  This is a direct consequence of Proposition~\ref{prop:polkernel}.
\end{proof}

\section{Comparison in the range \texorpdfstring{${(1, \infty)}$}{[1,infty)}}\label{sec:comparison}

\begin{thm}\label{thm:pqcompgeneral}
  Let $G$ be a finitely generated group of exponential growth
  and $p,q \in (1,\infty)$ with~$p < q$.
  \begin{enumerate}
  \item The inclusion~$C_*^p(G) \longrightarrow C_*^q(G)$ is a chain homotopy
    equivalence.
  \item In particular, the canonical map~$H_*^p(G) \longrightarrow H_*^q(G)$
      (see Remark~\ref{rem:p<qmap}) is an isomorphism.
  \end{enumerate}
\end{thm}

The proof of Theorem~\ref{thm:pqcompgeneral} is based on the following
basic chain-level result, which will be proved in Section~\ref{subsec_completing_proof_estimates}.

\begin{prop}\label{prop:EBconstruction}
  Let $G$ be a finitely generated group of exponential growth with finite generating set~$S$
  and let $p,q \in (1,\infty)$ with~$p < q$.  Then there exists a
  chain map~$E \colon C_*(G;\IC) \longrightarrow C_*(G;\IC)$ and a chain
  homotopy~$B$ between~$E$ and the identity with the following
  properties: For all~$k, n \in \IN$ there exist~$K \in \IR_{>0}$ and $m \in \IN$
  such that for all~$c \in C_k(G;\IC)$ we have
  \begin{align}
    \bigl\| E(c) \bigr\|_{n,p}^S
    & \leq K \cdot \bigl( \|c\|_{m,q}^S + \| \partial c \|_{m,q}^S \bigr)
    \label{eq:Enorm}
    \\
    \bigl\| \partial E(c) \bigr\|_{n,p}^S
    & \leq K \cdot \bigl( \|c\|_{m,q}^S + \| \partial c \|_{m,q}^S \bigr)
    \label{eq:dEnorm}
    \\
    \bigl\| B(c) \bigr\|_{n,p}^S
    & \leq K \cdot \bigl( \|c\|_{m,p}^S + \| \partial c \|_{m,p}^S \bigr)
    \label{eq:Bnormpp}
    \\ 
    \bigl\| \partial B(c) \bigr\|_{n,p}^S
    & \leq K \cdot \bigl( \|c\|_{m,p}^S + \| \partial c \|_{m,p}^S \bigr)
    \label{eq:dBnorm}
  \end{align}
\end{prop}

Taking Proposition~\ref{prop:EBconstruction} for granted, the proof of Theorem~\ref{thm:pqcompgeneral} is immediate:

\begin{proof}[Proof of Theorem~\ref{thm:pqcompgeneral}]
  We write~$i \colon C_*^p(G) \longrightarrow C_*^q(G)$ for the canonical
  inclusion map. 
  Let $E$ and $B$ be maps as provided by Proposition~\ref{prop:EBconstruction}.
  Estimates~\eqref{eq:Enorm} and \eqref{eq:dEnorm} show that $E$ extends
  to a continuous chain map
  \[ \overline E \colon C_*^q(G) \longrightarrow C_*^p(G).
  \]
  Similarly, the Estimates~\eqref{eq:Bnormpp} and \eqref{eq:dBnorm} (for~$p$ and
  for~$q$) show that
  $B$ extends to continuous chain homotopies
  \begin{align*}
    \overline B(p) \colon C_*^p(G) & \longrightarrow C_*^p(G) \\
    \overline B(q) \colon C_*^q(G) & \longrightarrow C_*^q(G)
  \end{align*}
  between $E \circ i$ and the identity on~$C_*^p(G)$ and between $i
  \circ E$ and the identity on~$C_*^q(G)$, respectively. Therefore,
  $i$ is a chain homotopy equivalence and thus induces an
  isomorphism~$H_*^p(G) \longrightarrow H_*^q(G)$ on homology.
\end{proof}

\subsection{Diffusion}

It remains to construct the maps~$E$ and~$B$ in
Proposition~\ref{prop:EBconstruction}. The fundamental observation is
that $\ell^p$-norms on~$C_*(G;\IC)$ can be decreased by diffusing the
coefficients over a large number of simplices. Therefore, we diffuse
the simplices by coning them off with cone points in annuli of suitable
radii (Figure~\ref{fig:diffusion}). 

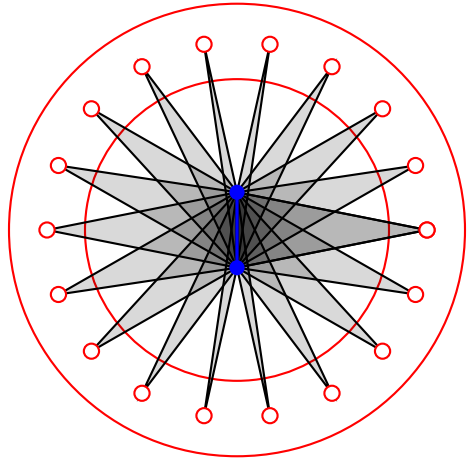
\begin{figure}
  \begin{center}
    \begin{tikzpicture}[x=1cm,y=1cm,thick]
      \draw[red] (0,0) circle (2);
      \draw[red] (0,0) circle (3);
      \foreach \j in {0,20,...,360} {
        \filldraw[fill=black, fill opacity=0.15] (0,-0.5) -- (\j:2.5) -- (0,0.5) -- cycle;
        \begin{scope}[red]
          \egvertx{(\j:2.5)}
        \end{scope}
      }
      \begin{scope}[blue, very thick]
        \draw (0,-0.5) -- (0,0.5);
        \gvertx{(0,-0.5)}
        \gvertx{(0,0.5)}
      \end{scope}
    \end{tikzpicture}
  \end{center}

  \caption{Diffusing chains/cones of a simplex (solid vertices; centre) using cone points
    (empty) in an annulus}
  \label{fig:diffusion}
\end{figure}

\begin{defn}[diffusion cone operator]\label{def:diffusion}
  Let $G$ be a finitely generated group, $k \in \IN$ and let~$Z \colon G^*
  \longrightarrow P_{\fin}(G)$ be a map (here $P_{\fin}(G)$ denotes the collection of finite subsets of $G$). The \emph{diffusion cone
    operator associated with~$Z$} is defined by
  \begin{align*}
    C_k(G;\IC) & \longrightarrow C_{k+1}(G;\IC)
    \\
      [e,g_1,\dots,g_k] & \longmapsto
      \frac1{|Z_{(g_1,\dots, g_k)}|} \cdot \sum_{z \in Z_{(g_1,\dots,g_k)}} [z,e,g_1,\dots,g_k]\,.\qedhere
  \end{align*}
\end{defn}

The key parameter of the diffusion cone construction is the
function~$Z$ determining the supports of the diffused simplices. We
will use wide enough annuli of large enough radii. More precisely,
we let the radii grow polynomially (of high degree) in terms of
the diameter of the original simplices.

\begin{defn}[diffusion annuli]\label{def:annuli}
  Let $G$ be a finitely generated group with a chosen finite generating set~$S$
  and let $N \in \IN_{>10}$. We define the \emph{diffusion annuli map of degree~$N$} for $k \in \IN$ by
  \begin{align*}
    Z \colon 
    G^k & \longrightarrow P_{\fin}(G) \\
    g & \longmapsto \overline Z_{\diam_S(g)},
  \end{align*}
  where 
  \begin{align*}
    \overline Z \colon \IN & \longmapsto P_{\fin}(G) \\
    r & \longmapsto \overline Z _r := 
    \bigl\{ g \in G \bigm| r^N - r^{N/10} < d_S(e,g) \leq r^N
    \bigr\}.
  \end{align*}
  Moreover, we write 
  \begin{align*}
    \varphi \colon \IN & \longrightarrow \IN \\
    r & \longmapsto 2 \cdot r^N.
    \qedhere
  \end{align*}
\end{defn}

Before starting with the actual proof of Proposition~\ref{prop:EBconstruction},
we first collect some basic estimates concerning this diffusion construction:

\begin{lem}[accumulation control]\label{lem:accum}
  In the situation of Definition~\ref{def:annuli}, we have for all~$k \in \IN_{\geq 1}$,
  all~$g \in G^k$, and all~$z \in Z_g$:
  \begin{enumerate}
  \item\label{item_diam_estimate} Clearly, $\diam_S [z,e,g_1,\dots,g_k] \leq \varphi(\diam_S(g))$.
  \item If $j \in \{1,\dots,k\}$, $h \in G^k$, and $w \in Z_h$ satisfy the
    relation 
    $[z,e,g_1,\dots,\widehat g_j,\dots,g_k]
      = [w,e,h_1,\dots,\widehat h_j,\dots,h_k]$, then
    \[ w = z
       \quad\text{and}\quad
       \diam_S(g) = \diam_S(h).  
    \]
  \item If $h \in G^k$ and $w \in Z_h$ satisfy $[z,g_1,\dots,g_k] =[w,h_1,\dots,h_k]$, then
    \[ w = h_1 \cdot g_1^{-1} \cdot z
       \quad\text{and}\quad
       \diam_S(g) = \diam_S(h).
    \]
  \end{enumerate}
\end{lem}
\begin{proof}
  \emph{Ad~1.}
  This is immediate from the construction.

  \emph{Ad~2.}
  Because both simplices have the same $1$-vertex (namely~$e$), all the
  vertices must coincide. Thus, $w = z$. Because the annuli defined by~$\overline Z$
  are disjoint for different radii and because $w = z \in Z_h \cap Z_g$, we obtain
  $\diam_S(g) = \diam_S(h)$.

  \emph{Ad~3.}
  The assumption implies that
  \[ \fa{j \in \{1,\dots,k\}} z^{-1} \cdot g_j = w^{-1} \cdot h_j.
  \]
  In particular, $w = h_1 \cdot g_1^{-1} \cdot z$.
  Using the abbreviations $r_g := \diam_S(g)$ and $r_h := \diam_S(h)$,
  we obtain by the triangle inequality that  $d_S(e,z^{-1}\cdot g_1)
    = d_S(e,w^{-1}\cdot h_1)$ is in the intersection 
  \begin{align*}
    [r_g^N - r_g^{N/10} - r_g, r_g^N + r_g]
    \cap 
    [r_h^N - r_h^{N/10} - r_h, r_h^N + r_h]
  \end{align*}
  (which is hence non-empty). 
  Therefore, $r_g = r_h$.
\end{proof}

\begin{lem}[norm control]\label{lem:normcontrol}
  In the situation of Definition~\ref{def:annuli}, let $k \in \IN$, and let 
  \begin{align*}
    I_k := \bigl\{ (z,g_1,\dots,g_k) \bigm| (g_1,\dots, g_k) \in G^k,\ z \in Z_g \bigr\}.
  \end{align*}
  Furthermore, let $J_k$ be a set, let $\pi \colon I_k \longrightarrow J_k$ be a
  map, and let $\beta \colon I_k \longrightarrow \IN$ be a function controlling the
  size of the fibres of~$\pi$, i.e.,
  \[ \fa{i \in I_k} \bigl| \pi^{-1}(\pi(i))\bigr| \leq \beta(i).
  \]
  For functions~$f \colon I_k \longrightarrow \IC$, we define the push-forward
  \begin{align*}
    \pi_* f \colon J_k & \longrightarrow \IC \\
    j & \longmapsto \sum_{i \in \pi^{-1}(j)} f(i).
  \end{align*}
  Finally, let $p \in (1,\infty)$ and let $f \colon I_k
  \longrightarrow \IC$ be a function with finite support.
  \begin{enumerate}
  \item Then
    \[ \| \pi_* f\|_p \leq \biggl(
       \sum_{i \in I_k} \beta(i)^p \cdot \bigl|f(i) \bigr|^p \biggr)^{1/p}.
    \]
  \item
    Moreover, let $q \in (1,\infty)$ with~$p< q$, let $q' \in
    (1,\infty)$ with~$1/q + 1/q' = 1/p$,
    and let $w \colon I_k \longrightarrow \IC$
    be a function such that $\pi_* w \colon J_k \longrightarrow \IC$ is $q'$-summable.
    Then (with respect to pointwise multiplication)
    \begin{align*}
      \bigl\| \pi_*(f \cdot w) \bigr\|_p
      & \leq \bigl\| \pi_* f \bigr\|_q \cdot \|\pi_*w\|_{q'}.
    \end{align*}
  \end{enumerate}
\end{lem}
\begin{proof}
  The first part is a consequence of the following elementary estimate: For
  all~$n \in \IN$ and all~$a_1,\dots,a_n \in \IC$, we have (by looking at a
  coefficient of maximal modulus) 
  \[ |a_1 + \dots + a_n|^p \leq n^p \cdot \bigl( |a_1|^p + \dots + |a_n|^p\bigr).
  \]
  The second part is just an instance of the generalized H\"older inequality.
\end{proof}

\subsection{Completing the proof of the comparison result}
\label{subsec_completing_proof_estimates}

\begin{proof}[Proof of Proposition~\ref{prop:EBconstruction}]
  We choose the parameter~$N := 100$ for the construction
  in Definition~\ref{def:annuli} (basically any choice will work
  because of the exponential growth of~$G$).
  Let $Z \colon G^* \longrightarrow P_{\fin}(G)$ be the associated
  diffusion annuli map (Definition~\ref{def:annuli}) and let $B \colon
  C_*(G;\IC) \longrightarrow C_{*+1}(G;\IC)$ be the diffusion cone operator
  associated with~$Z$ (Definition~\ref{def:diffusion}). We then
  define
  \begin{align*}
    E := \id - \partial \circ B - B \circ \partial
    \colon C_*(G;\IC) \longrightarrow C_*(G;\IC).
  \end{align*}
  It is clear
  that $E$ is a chain map and $B$ a chain homotopy between~$E$
  and the identity on~$C_*(G;\IC)$.

  Therefore, it remains to prove the norm estimates. We first replace 
  this zoo of estimates by the following estimates: For all~$k,n \in \IN$
  there exist~$K \in \IR_{>0}$ and $m \in \IN$ such that for all~$c \in C_k(G;\IC)$
  we have
  \begin{align}
    \bigl\| B(c) \bigr\|_{n,p}^S
    & \leq K \cdot \|c\|_{m,p}^S 
    \label{eq:Bnormppp}
    \\ 
    \bigl\| B(c) \bigr\|_{n,p}^S
    & \leq K \cdot \|c\|_{m,q}^S 
    \label{eq:Bnormpq}
    \\
    \bigl\| c - \partial B(c) \bigr\|_{n,p}^S
    & \leq K \cdot \|c\|_{m,q}^S 
    \label{eq:iddBnorm}
  \end{align}
  These estimates imply the Estimates~\eqref{eq:Enorm}--\eqref{eq:dBnorm} (modulo
  unification of the constants by taking the maximum) as follows:
  \begin{itemize}
  \item Estimate~\eqref{eq:Bnormpp} follows from Estimate~\eqref{eq:Bnormppp}.
  \item Estimate~\eqref{eq:Enorm} follows from the fact that $E = \id -
    \partial \circ B - B \circ \partial$ and the Estimates~\eqref{eq:iddBnorm} and
    \eqref{eq:Bnormpq} (with modified constants).
  \item Estimate~\eqref{eq:dEnorm} follows from~\eqref{eq:Enorm} and the
    fact that $E$ is a chain map.
  \item Estimate~\eqref{eq:dBnorm} then follows from~$\partial \circ B = \id - E - B \circ \partial$
    and the Estimates~\eqref{eq:Enorm} (and Remark~\ref{rem:p<qmap}), and \eqref{eq:Bnormpp}
    (with modified constants).
  \end{itemize}

  In the following, let~$k, n\in \IN$, and let
  \[ c= \sum_{g \in G^k} a_g \cdot [e,g_1, \dots, g_k] \in C_k(G;\IC).
  \]

  We will first prove~\eqref{eq:Bnormppp};
  of course, \eqref{eq:Bnormppp} follows from~\eqref{eq:Bnormpq} (with Remark~\ref{rem:p<qmap}),
  but we will use this straightforward estimate as warm-up for the other estimates.
  By construction of the
  diffusion cone operator~$B$, we have (using Lemma~\ref{lem:accum}.\ref{item_diam_estimate} for the first inequality, and Definition~\ref{def:annuli} of $\varphi$ for the second inequality)
  \begin{align*}
    \bigl\| B_k(c) \bigr\|_{n,p}^S
    & = \biggl\| \sum_{g \in G^k} a_g \cdot \frac1{|Z_g|} \cdot
    \sum_{z \in Z_g} [z,e,g_1,\dots,g_k] \biggr\|_{n,p}^S
    \\
    & = \biggl(
    \sum_{g \in G^k} \sum_{z\in Z_g} \frac1{|Z_g|^p} \cdot |a_g|^p \cdot
    \bigl(\diam_S[z,e,g_1,\dots,g_k]\bigr)^n
    \biggr)^{1/p}
    \\
    & \leq \biggl(
    \sum_{g \in G^k} \frac1{|Z_g|^{p-1}} \cdot |a_g|^p
    \cdot \varphi\bigl(\diam_S(g)\bigr)^n
    \biggr)^{1/p}
    \\
    & \leq \biggl(
    \sum_{g \in G^k} |a_g|^p
    \cdot 2^n \cdot \bigl(\diam_S(g)\bigr)^{N \cdot n}
    \biggr)^{1/p}
    \\
    & = 2^{n/p} \cdot \|c\|_{N\cdot n,p}. 
  \end{align*}

  Before proving~\eqref{eq:Bnormpq} and \eqref{eq:iddBnorm},
  let us first fix some notation:
  Because of~$q > p$, there is some~$q'$ such that $1/q +
  1/q' = 1/p$; let $x := 1/q$, let $y := 1 - x = 1-1/p+1/q'$, and let
  $\varepsilon := y \cdot q' - 1 = q' \cdot (1 - 1/p) > 0$.   
  
  Let us establish~\eqref{eq:Bnormpq} (with~$m = 1$): 
  The generalized H\"older inequality shows that
  \begin{align*}
    \bigl\| B_k(c) \bigr\|_{n,p}^S
    & = \biggl(
    \sum_{g \in G^k} \sum_{z\in Z_g} \frac1{|Z_g|^p} \cdot |a_g|^p \cdot
    \bigl(\diam_S[z,e,g_1,\dots,g_k]\bigr)^n
    \biggr)^{1/p}
    \\
    & \leq \biggl(
    \sum_{g \in G^k} \sum_{z\in Z_g}
    \frac1{|Z_g|^{q\cdot x}} \cdot |a_g|^q \cdot \diam_S(g)
    \biggr)^{1/q}
    \\
    & \qquad \cdot
    \biggl(
    \sum_{g \in G^k} \sum_{z\in Z_g}
    \frac1{|Z_g|^{q'\cdot y}} \cdot
    \frac{\bigl(\diam_S[z,e,g_1,\dots,g_k]\bigr)^{q'\cdot n / p}}
         {\bigl(\diam_S(g)\bigr)^{q'/ q}}
    \biggr)^{1/q'}
  \end{align*}
  We denote the first factor by~$A_1$ and the second factor by~$A_2$.
  As $q \cdot x =1$, we obtain
  \begin{align*}
    A_1 & = \biggl(
    \sum_{g \in G^k} |Z_g| \cdot 
    \frac1{|Z_g|^{q\cdot x}} \cdot |a_g|^q \cdot \diam_S(g)
    \biggr)^{1/q}
    = \|c\|_{1,q}^S.
  \end{align*}
  The term~$A_2$ can be estimated via
  \begin{align*}
    A_2^{q'} & \leq 
    \sum_{g \in G^k} \sum_{z\in Z_g}
    \frac1{|Z_g|^{q'\cdot y}} \cdot
    \frac{\varphi\bigl(\diam_S(g)\bigr)^{q'\cdot n / p}}
         {\bigl(\diam_S(g)\bigr)^{q'/ q}}
    \\
    & \leq 
    \sum_{g \in G^k} 
    \frac1{|Z_g|^{q'\cdot y - 1}} \cdot
    \frac{\varphi\bigl(\diam_S(g)\bigr)^{q'\cdot n / p}}
         {\bigl(\diam_S(g)\bigr)^{q' / q}}
    \\
    & \leq \sum_{r \in \IN}
    \frac{\bigl| \{ g\in G^k \mid \diam_S\{e,g_1,\dots,g_k\} = r\} \bigr|}{|\overline Z_r|^\varepsilon} \cdot
    \frac{\varphi(r)^{q'\cdot n/p}}{r^{q'/q}}
    \\
    & \leq \sum_{r \in \IN}
    \frac{\beta_{G,S}(r)^k}{|\overline Z_r|^\varepsilon} \cdot
    \frac{\varphi(r)^{q'\cdot n/p}}{r^{q'/q}},
  \end{align*}
  where $\beta_{G,S} \colon \IN \longrightarrow \IN$ denotes the growth
  function of~$G$ with respect to~$S$.
  The second factor in the series above is dominated by a polynomial (in~$r$);
  we will now show that the first factor decreases exponentially in~$r$:
  By definition, we have 
  \[ \beta_{G,S}(r)^k \leq \bigl(4 \cdot |S| \bigr)^{r \cdot k}.
  \]
  Because $G$ has exponential growth, there is an~$\alpha \in \IR_{>1}$
  such that $\beta_{G,S}(r) \geq \alpha^r$ holds for all~$r \in \IN$.
  Therefore, for all~$r \in \IN$, 
  \[ |\overline Z _r|
     \geq \bigl|\beta_{G,S}(r^{N/10}/2)| 
     \geq \alpha^{r^{N/10}/2},
  \]
  and so
  \begin{align*}
    \frac{\beta_{G,S}(r)^k}{|\overline Z_r|^\varepsilon}
    & \leq \frac{\bigl(4 \cdot |S|\bigr)^{r \cdot k}}{\alpha^{\varepsilon/2 \cdot r^{N/10}}},
  \end{align*}
  which (eventually) decreases exponentially in~$r$. Hence, $A_2^{q'}$ is dominated by a convergent
  series (whose value is independent of~$c$). This shows Estimate~\eqref{eq:Bnormpq}.
  
  Finally, we prove the most delicate Estimate~\eqref{eq:iddBnorm}. By construction,
  \begin{align*}
    c - \partial & B_k(c) \\
    & = \sum_{g \in G^k} a_g \biggl( [e,g_1,\dots,g_k]
    - \frac1{|Z_g|} \cdot \sum_{z \in Z_g} \partial [z,e,g_1,\dots,g_k]\biggr)
    \\
    & = \sum_{g \in G^k} a_g \cdot \frac1{|Z_g|} \cdot \sum_{z \in Z_g} \biggl( -[z,g_1,\dots,g_k]
        + \sum_{j=1}^k (-1)^{j+1} \cdot [z,e,g_1,\dots,\widehat g_j, \dots, g_k] \biggr).
  \end{align*}
  Therefore,
  \begin{align*}
    \bigl\|     c - \partial B_k(c) \bigr\|_{n,p}^S
    & \leq \biggl\| \sum_{g\in G^k} \frac1{|Z_g|} \cdot a_g \cdot \sum_{z \in Z_g} [z,g_1,\dots,g_k]\biggr\|_{n,p}^S
    \\
    & \qquad + \sum_{j=1}^k \biggl\| \sum_{g\in G^k} \frac1{|Z_g|} \cdot a_g \cdot
    \sum_{z \in Z_g} [z,e,g_1,\dots,\widehat g_j,\dots, g_k]\biggr\|_{n,p}^S.
  \end{align*}
  We will treat these $k+1$~sums separately. 
  In order to introduce~$\|-\|_{m,q}$, we again will use the generalized H\"older
  inequality. However, in contrast with the previous estimates, we now have to 
  carefully control accumulations of coefficients
  on $k$-simplices (using Lemma~\ref{lem:accum} and Lemma~\ref{lem:normcontrol}).
    
  We will only treat the first sum in detail (the other sums can be handled in
  the same way by modifying~$J_k$ accordingly).
  We will apply Lemma~\ref{lem:normcontrol} to the following situation:
  We consider the set
    \[ J_k := \bigl\{ [z,g_1,\dots,g_k] \bigm| (z,g_1,\dots,g_k) \in I_k \bigr\}
       \subset C_k(G;\IC),
    \]
    together with the canonical projection~$\pi \colon I_k \longrightarrow J_k$.    
    In view of Lemma~\ref{lem:accum}, the projection~$\pi$ has $\beta$-controlled
    fibres, where
    \begin{align*}
      \beta \colon I_k & \longrightarrow \IN \\
      (z,g_1,\dots,g_k) & \longmapsto \beta_{G,S} \bigl(\diam_S (g) \bigr)^k.
    \end{align*}
    Let $\delta \in \IR_{>0}$ with~$\delta < y - 1/q' = \min(y, \varepsilon/q')$ and 
    \begin{align*}
      f \colon I_k & \longrightarrow \IC \\
      (z,g_1,\dots,g_k) & \longmapsto
      \frac1{|Z_g|^{x + \delta}} \cdot a_g \cdot \diam_S[z,g_1,\dots,g_k]^{1/q}\,,
      \\
      w \colon I_k & \longrightarrow \IC \\
      (z,g_1,\dots,g_k) & \longmapsto
      \frac1{|Z_g|^{y-\delta}} \cdot \diam_S[z,g_1,\dots,g_k]^{n/p - 1/q}\,.
    \end{align*}
  Then, by construction,
  \[ \biggl\| \sum_{g \in G^k} \frac1{|Z_g|} \cdot a_g \cdot
     \sum_{z \in Z_g} [z,g_1,\dots,g_k] \biggr\|_{n,p}^S
     = \bigl\|\pi_* (f \cdot w) \bigr\|_p. 
  \]
  We will now bound~$\|\pi_*(f \cdot w)\|_p$ from above with the help of
  Lemma~\ref{lem:normcontrol}: Clearly, $f$ has finite support. Let us show that~$\pi_* w$ is a $q'$-summable function. By definition of~$w$, we have
  (with~$\Phi(r) := 2 \cdot r^{N \cdot (n/p-1/q) \cdot q'}$) 
  \begin{align*}
    \sum_{i \in I_k} \beta(i)^{q'} \cdot \bigl|w(i)\bigr|^{q'}
    & \leq \sum_{(z,g) \in I_k} \frac{\beta(z,g)^{q'}}{|Z_g|^{(y-\delta)\cdot q'}} \cdot \Phi(\diam_S(g))
    \\
    & \leq \sum_{r\in \IN} \beta_{G,S}(r)^k \cdot |\overline Z_r| \cdot \frac{\beta_{G,S}(r)^{q'\cdot k}}{|\overline Z_r|^{(y-\delta) \cdot q'}}
      \cdot \Phi(r)
    \\
    & \leq \sum_{r\in \IN} \frac{\beta_{G,S}(r)^{k+q'\cdot k}}
                              {|\overline Z_r|^{(y-\delta) \cdot q' -1}}
      \cdot \Phi(r).
  \end{align*}
  Because $(y- \delta) \cdot q' -1 > 0$, the same argument as in the proof
  of Estimate~\eqref{eq:Bnormpq} shows that first factor (eventually) decreases at least
  exponentially in~$r$ while $\Phi$ grows only polynomially in~$r$. Therefore,
  this series is convergent; let $A_2$ be the value of this series.
  The first part of Lemma~\ref{lem:normcontrol}
  shows that $\pi_* w$ is $q'$-summable and that
  \[ \| \pi_* w \|_{q'} \leq A_2^{1/q^\prime}.
  \]
  Therefore, the second part of Lemma~\ref{lem:normcontrol} shows that
  \begin{align*}
    \bigl\|\pi_*(f \cdot w)\bigr\|_p
    & \leq \| \pi_* f \|_q \cdot \|\pi_* w\|_{q'}
    \leq A_2^{1/q^\prime} \cdot \| \pi_* f\|_q.
  \end{align*}
  It hence remains to provide a suitable estimate for~$\|\pi_*f\|_q$. 
  Using Lemma~\ref{lem:normcontrol}, we obtain
  \begin{align*}
    \| \pi_*f\|_q^q
    & \leq \sum_{i \in I_k} \beta(i)^q \cdot \bigl|f(i)\bigr|^q
    \\
    & \leq \sum_{(z,g) \in I_k}
    \frac{\beta_{G,S}(\diam_S (g))^{q \cdot k}}
         {|Z_g|^{q \cdot x + q \cdot \delta}} \cdot |a_g|^q \cdot \varphi(\diam_S g)
    \\
    & \leq 2 \cdot \sum_{g \in G^k}
    \frac{\beta_{G,S}(\diam_S (g))^{q \cdot k}}{|Z_g|^{1 + q \cdot \delta}} \cdot |Z_g| \cdot |a_g|^q \cdot \diam_S(g)^N
    \\
    & = 2 \cdot \sum_{g \in G^k}
    \frac{\beta_{G,S}(\diam_S (g))^{q \cdot k}}{|Z_g|^{q \cdot \delta}} \cdot |a_g|^q \cdot \diam_S(g)^N\,.
  \end{align*}
  Again, because $q \cdot \delta > 0$, we see as in the proof of the Estimate~\eqref{eq:Bnormpq}
  that the first factor is bounded, say by~$A_1$. Then,
  \[ \|\pi_* f\|_q^q
  \leq 2 \cdot A_1 \cdot \sum_{g \in G^k} |a_g|^q \cdot \diam_S(g)^N
  = 2 \cdot A_1 \cdot \|c\|_{N,q}^q\,.
  \]
  This completes the proof of Proposition~\ref{prop:EBconstruction}
  and hence of Theorem~\ref{thm:pqcompgeneral}.
\end{proof}

\section{A vanishing result in degree~\texorpdfstring{$1$}{1}}\label{sec:vanishing}

We have the following vanishing result for the free group~$F_2$ of rank~$2$:

\begin{thm}\label{thm:f2vanishing}
  Let $p \in (1,\infty)$. Then the canonical homomorphism~$H_1(F_2;\IC)
  \longrightarrow H_1^p(F_2)$ is the zero map. 
\end{thm}

\begin{proof}
  Let $S := \{\alpha,\beta\}$ be a free generating set of the free group~$F_2$ of
  rank~$2$. In this proof, all distances, diameters, norms, etc.\ will be taken with
  respect to this generating set~$S$.

  Before starting with the actual proof, we perform the following
  reductions:
  \begin{itemize}
  \item
    In view of Theorem~\ref{thm:pqcompgeneral}, it suffices to prove
    Theorem~\ref{thm:f2vanishing} for~$p>2$. 
  \item
    Because~$H_1(F_2;\IC)$ is generated by the homology classes corresponding
    to the cycles~$[e,\alpha]$ and $[e,\beta]$, it suffices to show that
    the classes in~$H^p_1(F_2)$ represented by~$[e,\alpha]$ and $[e,\beta]$ are trivial.
  \item
    Since the classes represented by~$[e,\alpha]$ and $[e,\beta]$ only differ by an
    isometric automorphism of~$F_2$, it suffices to prove the vanishing for~$[e,\alpha]$. 
  \end{itemize}
  To this end, we will construct an explicit chain~$b$ in~$C^p_2(F_2)$ whose
  boundary is~$[e,\alpha]$.
  
  The geometric idea for the construction of such a $2$-chain~$b$ is
  to start with two $2$-simplices with coefficient~$1/2$ that
  contain~$[e,\alpha]$ as an edge; inductively, we then choose two
  $2$-simplices with halved coefficients that contain the new
  edges~\dots\ (Figure~\ref{fig:f2vanishing}).
  The resulting infinite chain will converge in the $\ell^p$-setting
  because the coefficients are distributed over enough summands. 
  The main technical difficulty is to ensure that the
  weights are really distributed so that they do not accumulate on
  simplices via accidental cancellations. This will be
  achieved by a careful selection of markers and suffixes that encode
  the induction level and the two different choices at each stage.

  \begin{figure}
    \begin{center}
      \begin{tikzpicture}[thick]
        \filldraw[fill=black,fill opacity=0.1] (0,0) -- (0,2) -- (4,0) -- cycle;
        \filldraw[fill=black,fill opacity=0.1] (0,0) -- (0,2) -- (4,2) -- cycle;
        \gvertx{(4,0)} 
        \gvertx{(4,2)} 
        \gvertx{(0,0)} 
        \gvertx{(0,2)} 
        \draw (-0.5,0) node {$e$};
        \draw (-0.5,2) node {$x$};
        \draw (4.3,0) node[anchor=west] {$xm(x)t_d$};
        \draw (4.3,2) node[anchor=west] {$xm(x)s_d$};
        \draw (2,-0.5) node {$t(x)$};
        \draw (2,2.5) node {$s(x)$};
      \end{tikzpicture}
    \end{center}

    \caption{For each edge~$[e,x]$, we choose two $2$-simplices that
      contain this edge (and halve the coefficients).}
    \label{fig:f2vanishing}
  \end{figure}
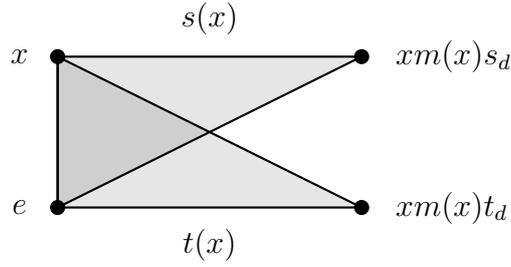
  
  We will describe the construction of~$b$ in a top-down manner, first giving
  the final formula and then explaining all the ingredients: For~$D\in \IN$, we
  set
  \[ b(D) := \sum_{d=0}^D \sum_{x \in W(d)} \frac1{2^{d+1}} \cdot
             \varepsilon(x) \cdot \bigl(s(x) + t(x)\bigr) \in C_2(F_2;\IC).
  \]
  We will then show that the sequence~$(b(D))_{D \in \IN}$ converges to a chain
  \[ b := \lim_{D \rightarrow \infty} b(D) \in C^p_2(F_2)
  \]
  that satisfies~$\partial b = [e,\alpha]$. But first we have to explain
  the ingredients of~$b(D)$: To this end, we define (by mutual recursion)
  the subsets~$W(d) \subset F_2$ (keeping track of the set of edges),
  the suffixes~$s_d, t_d \in F_2$, the markers~$m(x) \in F_2$
  and the $2$-simplices~$s(x)$ and~$t(x)$:
  \begin{itemize}
  \item For each~$d \in \IN$, we set~$s_d := \alpha^d\beta^d$ and $t_d := \beta^d \alpha^d \in F_2$.
  \item We set~$W(0) := \{\alpha\}$ (and $m(\alpha) := e$)
    and for~$d \in \IN_{\geq 1}$, we let
    \[ W(d) := \bigcup_{x\in W(d-1)} W(x,d) \subset F_2,
    \]
    where (for each~$x \in W(d-1)$)
    \[ W(x,d) := \{ x m(x) s_d, m(x)s_d, xm(x)t_d, m(x) t_d\} \subset F_2.
    \]
  \item
    Inductively, we see that $|W(d)| \leq 4^d$ for all~$d \in \IN$. We
    can thus choose an injection~$m \colon W(d) \longrightarrow
    \{\alpha, \beta\}^{2\cdot d}$ and view the words~$m(x)$ with~$x
    \in W(d)$ as elements of~$F_2$.
  \item
    For~$d \in \IN$ and $x \in W(d)$, we set
    \[ s(x) := \bigl[e, x, xm(x)s_d \bigr], \quad
       t(x) := \bigl[e, x, xm(x)t_d \bigr] \in C_2(F_2;\IC).
    \]
  \item Finally, the signs~$\varepsilon(\dots)$ are defined as follows:
    We set~$\varepsilon(\alpha) :=1$; for $d \in \IN_{>0}$
    and~$x \in W(d-1)$, we set
    \begin{align*}
      \varepsilon(m(x) s_d) & := - \varepsilon(x)
      \\
      \varepsilon(m(x) t_d) & := - \varepsilon(x)
      \\
      \varepsilon(xm(x) s_d) & := \varepsilon(x)
      \\
      \varepsilon(xm(x) t_d) & := \varepsilon(x).
    \end{align*}
  \end{itemize}
  By construction, all elements of~$W(d)$ consist of \emph{non-negative}
  powers of~$\alpha$ and $\beta$ and no cancellations occur in the
  definitions above. Therefore, $s$, $t$, and $\varepsilon$ are
  well-defined. Moreover, the construction of the edge sets~$W(d)$ is justified
  by the following observation: For each~$d \in \IN$ and each~$x \in W(d)$,
  we have
  \begin{align*}
    \partial \bigl( s(x) + t(x) \bigr)
    & = \partial \bigl( [e,x,xm(x)s_d] + [e,x,xm(x)t_d]\bigr)
    \\
    & = [e,m(x)s_d] - [e,xm(x)s_d] + [e,x]
      + [e,m(x)t_d] - [e,xm(x)t_d] + [e,x].
  \end{align*}

  In order to prove convergence of~$(b(D))_{D \in \IN}$ and $(\partial
  b(D))_{D \in \IN}$, we need to estimate the diameters of the
  simplices involved: For~$d \in \IN$ and $x \in W(d)$, we have
  \[ \diam s(x) \leq d(e,x) + 4 \cdot d,
  \quad
     \diam t(x) \leq d(e,x) + 4 \cdot d;
  \]
  inductively, we obtain for $x \in W(d)$
\[d(e,x) \in O(d^2)\]  
  and therefore 
  \[ \diam s(x),\quad \diam t(x) \in O(d^2).
  \]
  We now give the convergence arguments:
  \begin{itemize}
  \item The sequence~$(b(D))_{D \in \IN}$ is Cauchy with respect to~$\|-\|_{n,p}$: 
    Let $D, D' \in \IN$ with~$D' > D \geq 0$ and let $n \in \IN$. By construction, we
    have
    \begin{align*}
      b(D') - b(D)
      & = \sum_{d=D+1}^{D'} \sum_{x \in W(d)} \frac1{2^{d+1}} \cdot \varepsilon(x) \cdot \bigl(s(x) + t(x)\bigr).
    \end{align*}
    The markers/suffixes show that all of these $2$-simplices are different (so
    no cumulations of coefficients occur). Therefore,
    \begin{align*}
      \bigl\| b(D') - b(D) \bigr\|_{n,p}
      & \leq \sum_{d=D+1}^{D'} \frac{|W(d)|}{(2^{d+1})^p} \cdot O(d^{2n})
      \\
      & \leq \sum_{d=D+1}^{D'} \frac{4^d}{(2^{d+1})^p} \cdot O(d^{2n}).
    \end{align*}
    Because $p > 2$, the corresponding series on the right-hand side
    is convergent. Therefore, these differences between its partial
    sums form a Cauchy sequence.
  \item The sequence~$(\partial b(D))_{D \in \IN}$ is Cauchy with respect to~$\|-\|_{n,p}$: 
    Let $D ,D' \in \IN$ with~$D' > D \geq 0$ and let $n \in \IN$. By construction, we have
    \begin{align*}
      \partial b(D') - \partial b(D) 
      & = \sum_{d = D+1}^{D'} \sum_{x\in W(d)} \frac1{2^{d+1}} \cdot \varepsilon(x)
      \cdot \partial \bigl(s(x) + t(x)\bigr) 
      \\
      & = \sum_{x \in W(D+1)} \frac1{2^{D+1}} \cdot \varepsilon(x) \cdot [e,x] - \sum_{y \in W(D'+1)} \frac1{2^{D'+1}} \cdot \varepsilon(y) \cdot [e,y].
    \end{align*}
    The markers/suffixes show that all of these $1$-simplices are different (so
    no cumulations of coefficients occur). Therefore,
    \begin{align*}
      \bigl\| \partial b(D') - \partial b(D) \bigr\|_{n,p}
      & \leq \frac{|W(D' +1)|}{(2^{D'+1})^p} \cdot O(D'{}^{2n})
      + \frac{|W(D+1)|}{(2^{D+1})^p} \cdot O(D{}^{2n})
      \\
      & \leq \frac{4^{D'+1}}{(2^{D'+1})^p} \cdot O(D'{}^{2n})
      + \frac{4^{D+1}}{(2^{D+1})^p} \cdot O(D{}^{2n}).
    \end{align*}
    Because $p > 2$, these terms converge to~$0$ for~$D,D' \rightarrow \infty$.
  \end{itemize}

  Thus, we have established that $b = \lim_{D \rightarrow \infty} b(D) \in
  C^p_2(F_2)$ is a well-defined chain. By a similar computation as the previous one for~$\partial b(D') - \partial b(0)$, 
  we have
  \[ \partial b = [e,\alpha],
  \]
  as claimed.
\end{proof}

\bibliography{./poly_p}
\bibliographystyle{amsalpha}

\end{document}